\documentclass[12pt]{article}
\usepackage{braket}
\usepackage{amssymb}
\usepackage{amsmath,amsthm}
\usepackage[dvipdfm]{graphicx}
\usepackage{float}
\usepackage{url,bm}
\usepackage{subfigure}

\makeatletter
\newcommand{\xlabel}[2]{{\def\@currentlabel{#2}\label{#1}}#2}
\makeatother
\newcommand{\romanref}[1]{{\rm\ref{#1}}}

\newcommand{\RR}{\mathbb{R}}
\newcommand{\ZZ}{\mathbb{Z}}

\newcommand{\GL}{\operatorname{GL}}
\newcommand{\Int}{\operatorname{Int}}
\newcommand{\conv}{\operatorname{conv}}

\newcommand{\argmax}{\operatorname*{ArgMax}}
\newcommand{\vvert}{\operatorname*{vert}}

\theoremstyle{plain}   
\newtheorem{thm}{Theorem}[section]

\newtheorem{lemma}[thm]{Lemma}
\newtheorem{lem}[thm]{Lemma}

\newtheorem{proposition}[thm]{Proposition}
\newtheorem{prop}[thm]{Proposition}
\newtheorem{conjecture}[thm]{Conjecture}

\theoremstyle{remark}  
\newtheorem{remark}[thm]{Remark}

\newtheorem{example}[thm]{Example}
\newtheorem{ex}[thm]{Example}
\newtheorem*{acknowledgment}{Acknowledgment}

\theoremstyle{definition}  
\newtheorem{definition}[thm]{Definition}
\newtheorem{defi}[thm]{Definition}

\begin{document}

\title{Separation of integer points by a hyperplane under some weak notions of discrete convexity}
\author{Takuya Kashimura%
\thanks{Department of 
Mathematical Informatics, 
Graduate School of Information Science and Technology, 
University of Tokyo.},
Yasuhide Numata\footnotemark[1]\ \thanks{JST, CREST}  
and 
Akimichi Takemura\footnotemark[1]\ \footnotemark[2]}
\date{February 2010}
\maketitle

\begin{abstract}
We give some sufficient conditions of separation of two
sets of integer points by a hyperplane.  Our conditions are
related to the notion of convexity of sets of integer points and 
are weaker than existing notions.
\end{abstract}

\section{Introduction}
In this paper,
for a given nonempty set $S$ of integer points in $\RR^d$,
we investigate conditions on separation of two
sets $A \subset S$ and $B = S \setminus A$
by an affine hyperplane. Consideration of hyperplane separation naturally leads to various
notions of discrete convexity of $A$ and $B$.   The hyperplane separation theorem
for two disjoint convex sets in  $\RR^d$, 
or the Hahn-Banach separation theorem for infinite dimensional spaces, 
is a fundamental fact and
is a logical basis for various fields such as optimization or game theory.
However, because of the discreteness, separation of sets of integer
points is a subtle
problem.  Murota (\cite{murota98,murota03,murota09}) and his collaborators have developed 
the whole new field of
``discrete convex analysis'' and the hyperplane separation is an important
motivation for the field.

In Murota's works and works by earlier authors, hyperplane separation results
have been proved 
by imposing some nice conditions of discrete convexity of $A$ and $B$.
However for some problems we want to know whether separation holds under weaker conditions
on $S$, $A$ and $B$.
In \cite{hara} we needed a separation result for $\RR^2$ under ``parallelogram condition'' 
given in Definition \ref{def:Parallelogram} below.  
A similar result 
for higher dimension was conjectured in \cite{hara}
and the conjecture is the motivation for the present study.
As shown in Section \ref{sec:finite} we found that there is a 
large gap between $\RR^2$ and $\RR^d$ for $d \ge 3$.

Non-convexity of a set $A$ of integer points is usually characterized by
existence 
of a ``hole'' in $A$.  Recently active research is conducted on various notions
of holes. 
Fano polytopes, which are motivated from algebraic geometry, are
actively investigated from combinatorial viewpoint \cite{kas06,kas08}.  
They are rich sources of polytopes 
with a single hole in its interior.
Empty lattice simplices, whose only integer points are its vertices, 
have been studied in many contexts (see \cite{sebo} and references therein).
In the field of commutative algebra and its application to 
algebraic statistics, holes in a semigroup generated
by a set of integer points are 
of great interest (e.g.\ \cite{hty, Ohsugi-Hibi-ample}).

The organization of this paper is as follows.
In Section \ref{sec:definitions} 
we give definitions and some preliminary facts.
In Section \ref{sec:inifinite} 
we prove that hyperplane separation holds
under very weak condition for the case $S=\ZZ^d$.
In Section \ref{sec:finite}  
we consider finite $A$ and $B$ and
prove separation results when $S\subset \ZZ^2$ is integrally convex 
and when $S\subset \ZZ^2$ is hole free.  
For $\ZZ^d$, $d\ge 3$, we give some counterexamples.
We end the paper with discussion of some open problems 
in Section \ref{sec:conclusion}.

\section{Definitions and some preliminary facts}
\label{sec:definitions}

For a given nonempty set $S$ of integer points in $\RR^d$,
we consider some conditions which concern separation of two
nonempty sets 
$A \subset S$ and $B = S \setminus A$
by an affine hyperplane $H$.  
Throughout this paper we assume that $A,B$ are nonempty disjoint sets
and we denote $S=A\cup B$.
$H$ may be defined by a linear form with irrational coefficients.

We allow 
$H$ to contain points of both $A$ and $B$.
However on $H$ we are again concerned on separation of $A$ and $B$
by an affine space of codimension 1 within  $H$.  Therefore we make
the following definition.

\begin{defi}
Let $L$ be an affine subspace in $\RR^d$ of dimension $l$.
We call an affine subspace $P\subset L$ of dimension $l-1$ 
a {\em hyperplane}
in $L$. We say that {\em $P$ separates $A$  and $B$ in $L$}
if there exist 
two disjoint connected components $R^+$ and $R^-$ of $L\setminus P$
such that 
 \begin{align*}
 R^+ \cup R^- &=        L \setminus P, & 
 R^+ \cap S   &\subset A, &
 R^- \cap S   &\subset B.
 \end{align*}
\end{defi}

\begin{example}
\label{ex}
We consider the case when $A=\Set{0,1,2}$, 
$B=\Set{-1,-2}$. 
Let us take 
the rays $\RR_{>0}$ and $\RR_{<0}$ as $R^+$ and $R^-$, respectively.
Then  $P=\RR\setminus(R^+\cup R^-)$ has only the point $0\in A$.
In this case, $S \cap R^+=\Set{1,2}\subset A$ and 
$S\cap R^-=\Set{-1,-2}\subset B$.
Hence $P=\Set{0}$ separates $A$ and $B$ in $\RR$. 
\end{example}

\begin{defi}[Separation by a sequence of affine hyperplanes, Condition \xlabel{C:Hyperplanes}{H}]
Let  $H_i \ (i \ge 0)$ be an $i$-dimensional affine subspaces  of $\RR^d$ 
and  $H = \{ H_k, \dots, H_d \} \ (k \ge 0)$ 
be a sequence of affine hyperplanes such that
$H_{i-1} \subset H_i \ (i = k+1, \dots, d)$ and $H_d = \RR^d$. 
{\em $A$ and $B$ are separated by $H$},
if $H$ satisfies
\begin{enumerate}
\item $H_k \cap S \subset A$ or $H_k \cap S \subset B$, and 
\item for $k+1 \le i \le d$,  $H_{i-1}$ separates $A$ and $B$ in $H_i$.
\end{enumerate}
We say that $A$ and $B$ satisfy {\em Condition \romanref{C:Hyperplanes}}
if $A$ and $B$ are separated by some $H=\{ H_k, \dots, H_d \}$.
\end{defi}

We consider two examples of separation of $A \subset \ZZ^2$ 
and $B= \ZZ^2 \setminus A$.

\begin{ex}
\label{ex1}
Let $H_1$ be the line
defined by $x_2 = \varpi x_1$ for some irrational number $\varpi$,
$A=\{(x_1, x_2)\in \ZZ^2 | x_2 \ge \varpi x_1 \}$,
and $H_2=\RR^2$. Then 
$H=\{H_1, H_2\}$ separates $A$ and $B=\ZZ^2 \setminus A$.
(Moreover, for any $t$, $H'=\{ \{(t,\varpi t)\} ,H_1, H_2 \}$ separates
 $A$ and $B$.)
\end{ex}

\begin{remark}
For $S=\ZZ^d$
there may be no rational vector $p\in \ZZ^d$
such that a hyperplane $H = \Set{ x \in \RR^d | p_1 x_1 + \dots + p_d x_d = b }$ separates $A$ and $B$,
because $\ZZ^d$ is unbounded.
\end{remark}

\begin{ex}
\label{ex2}
Let $A = \{ (x_1,x_2) \in \ZZ^2 \ | \ x_1 > 0 \} 
\cup \{ (0,x_2) \in \ZZ^2 \ | \ x_2 \ge 0 \}$
as in Figure~\ref{fig1}.
If we set
\begin{align*}
	H_0 &= \{ (0,0) \}, & H_1 &= \{ (0,x_2) \, | \, x_2 \in \RR \}, &
	H_2 &= \RR^2,
\end{align*}
then $H = \{ H_0,H_1,H_2 \}$ separates $A$ and $B=\ZZ^2 \setminus A$.
Moreover, for each $t\in [-1,0]$, $\{ \{ (0,t) \} ,H_1,H_2 \}$  separates $A$ and $B$.
\begin{figure}
\begin{center}
	\includegraphics[clip, width = 0.35\hsize]{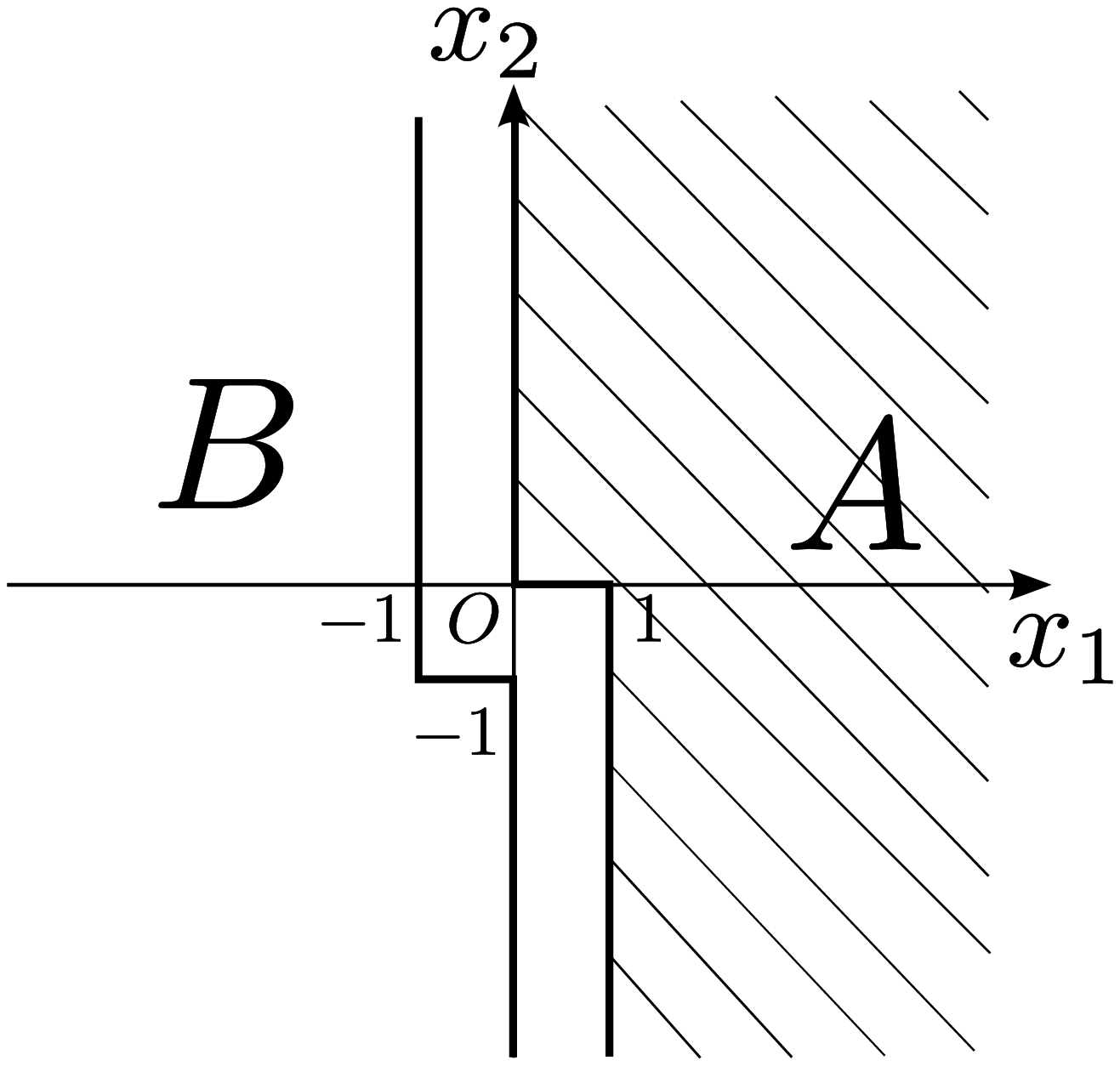}
\caption{Example \ref{ex2}}
\label{fig1}
\end{center}
\end{figure}
\end{ex}

Next we consider the following condition.
\begin{defi}[Ray condition, Condition \xlabel{C:Ray}{R}]
We say that $A$ and $B$ satisfy {\em Condition \ref{C:Ray}} if
for each line $L$ such that $A\cap L \neq \emptyset$ and $B\cap L \neq \emptyset$,
there exists a ray $L' \subset L$ such that $A \cap L = L' \cap S$.
\end{defi}
\begin{remark}
Condition \ref{C:Ray} is equivalent to the following:
for each line $L$ such that $A\cap L \neq \emptyset$ and $B\cap L \neq \emptyset$,
there exists $p\in L$ such that 
$\{ p \}$ separates $A$ and $B$ in $L$.
\end{remark}

Furthermore we consider the following condition.
\begin{defi}[Parallelogram condition, Condition \xlabel{C:Parallelogram}{P}]
\label{def:Parallelogram}
We say that $A, B$ satisfy {\em $k$-parallelogram condition} if
\begin{align*}
a_1, a_2, \dots, a_{k'} \in A, \ 
b_1, b_2, \dots, b_{k'} \in B \quad \Rightarrow \quad
\sum_{i=1}^{k'} a_i \neq \sum_{i=1}^{k'} b_i
\end{align*}
for all $k'\leq k$.
We call the $2$-parallelogram condition
{\em Condition \romanref{C:Parallelogram}}.
\end{defi}
\begin{remark}
By definition, the $k$-parallelogram condition implies 
the $(k-1)$-parallelogram condition.
\end{remark}
When we consider Condition  \romanref{C:Parallelogram},
we may have $a_1=a_2$.  
Then  the condition 
says that no point of $A$ is the mid-point of two points in $B$.
 Condition  \romanref{C:Parallelogram}
was 
considered in Hara et al.\ \cite{hara} in a statistical problem.
They showed that when $S=A\cup B$ is a $2$-dimensional rectangle, 
and if $A$ and $B$ are ``monotone'' and satisfy Condition 
\romanref{C:Parallelogram}, 
then there exists a line separating $A$ and $B$ 
(see Appendix E in Hara et al.\ \cite{hara}). 

We state the following basic fact on implications of the above 
conditions.

\begin{lem}
\label{lem:hyperplanes=>parallelogram}
Condition \romanref{C:Hyperplanes} implies Condition \romanref{C:Parallelogram}.
\end{lem}
\begin{proof}
Consider the case that at least one
point, say $a_1$, among $a_1,a_2 \in A, b_1,b_2 \in B$ 
belongs to an open half-space on one side of $H_{d-1}$. Then
$(a_1+a_2)/2$ also belongs to the open half-space and can not be
equal to $(b_1+b_2)/2$.  
If all of $a_1,a_2,b_1,b_2$ belong to $H_{d-1}$, 
then we can use the induction on dimension.
\end{proof}

By the same proof as 
in Lemma \ref{lem:hyperplanes=>parallelogram}, we have the following.
\begin{lemma} 
Condition \romanref{C:Hyperplanes} 
implies the $k$-parallelogram condition 
for every $k\ge 2$.
\end{lemma}

So far we have presented conditions concerning two sets $A$ and $B$. 
Now we give a condition of discrete convexity of a single $S$, which is also needed
for relating Condition \ref{C:Ray} to Condition \ref{C:Parallelogram}.

\begin{defi}
[$k$-convexity]
\label{def:line-convex}
$S \subset \ZZ^d$
is {\em $k$-convex} if $S$ satisfies the following:
\begin{align*}
a_1,\ldots,a_{k+1} \in S \Rightarrow \conv\set{a_1,\ldots,a_{k+1}} \cap
 \ZZ^d \subset S,
\end{align*}
where $\conv\set{a_1,\ldots,a_{k+1}}$ denotes the convex hull of
 $a_1,\ldots,a_{k+1}$ in $\RR^d$.
\end{defi}

\begin{remark}
If $S\subset \ZZ^d$ is $k$-convex, then $S$  is $(k-1)$-convex.
\end{remark}

\begin{lem}
\label{lem:parallelogram=>ray}
If $S=A\cup B$ is $1$-convex, then
Condition \romanref{C:Parallelogram}
implies
Condition \romanref{C:Ray}.
\end{lem}
\begin{proof}
 If $A\cap L\neq\emptyset$, $B\cap L \neq \emptyset$, then
 we can find $a \in A$ and $b \in B$ lying next to each other on 
the line $L$ by 1-convexity of $S$. 
 Assume that $a$ is on the right side of $b$, and let $L' \subset L$ denote the ray extending from $a$ not containing $b$. If there exists a member of $B$ in $L' \cap \ZZ^d$,
 then we find the nearest point to $b$ in $B$, denoted by $b'$. 
Let $a'$ denote the next left point of $b'$ in $L' \cap \ZZ^d$. Then we have $a' \in A$ and   
 $a + a' = b + b'$, a contradiction to Condition \ref{C:Parallelogram}. Thus $L' \cap \ZZ^d \subset A$.
 Also we have $(L \setminus L') \cap \ZZ^d \subset B$. Then it follows that $L' \cap \ZZ^d = A \cap L$.
 \end{proof}

 \begin{lem}
 \label{lem:c1=>c3=>c2}
 If $S$ is $1$-convex, then
 \begin{align*}
 \text{ Condition \romanref{C:Hyperplanes} }
 \Rightarrow
 \text{ Condition \romanref{C:Parallelogram} }
 \Rightarrow
 \text{ Condition \romanref{C:Ray}. }
 \end{align*}
 \end{lem}
 \begin{proof}
The lemma follows from Lemmas \ref{lem:hyperplanes=>parallelogram} and \ref{lem:parallelogram=>ray}.
\end{proof}

In Section \ref{sec:finite} we consider the case when $S$ is a hole free set or an integrally convex set. 
We recall its definition and basic properties (\cite[Section 3.4]{murota03}).
See also \cite{favat-tardella}.
\begin{definition}
\label{def:integrally}
$S\subset \ZZ^d$ is  {\em integrally convex}
if 
\begin{align*} 
x \in \conv (S) \Rightarrow x \in \conv (S\cap N(x)), 
\end{align*}
where $N(x)=\Set{y\in \ZZ^d| \text{$|x_i-y_i|<1$ for $i=1,\ldots,d $}}$
is the integral neighborhood of $x \in \RR^d$.
\end{definition}

This definition is also written as follows.

\begin{proposition}\label{prop:integrally2}
$S \subset \ZZ^d$ is integrally convex if and only if
\begin{align*} 
\conv(S \cap N(x)) = \conv(S) \cap \conv(N(x)) \quad 
(\forall x \in \RR^d). 
\end{align*}
\end{proposition}

\begin{definition}
$S\subset \ZZ^d$ is  {\em hole free}
if $S=\conv(S)\cap \ZZ^d$.
\end{definition}

\begin{proposition}
An integrally convex set is hole free.
\end{proposition}

\begin{remark}
For  $S\subset \ZZ^d$,
$S$ is $d$-convex if and only if $S$ is hole free. 
\end{remark}

\section{Separation of the whole integer lattice}
\label{sec:inifinite}
In this section, we 
consider the case $S = \ZZ^d$. We prove the following theorem.
\begin{thm} 
\label{thm:whole-lattice}
Let $S = \ZZ^d$.  Then for $A \subset S$ and $B = S \setminus A$,  Conditions \ref{C:Hyperplanes}, \ref{C:Ray}, \ref{C:Parallelogram} are all equivalent.
\end{thm}

By Lemma \ref{lem:c1=>c3=>c2} it suffices to prove Condition \ref{C:Ray}
$\Rightarrow$ Condition \ref{C:Hyperplanes}. 
We give the proof 
in a series of lemmas and the following key proposition.

\begin{prop}
\label{lemma:inductionstep}
Let $l<n \leq d$ and $S=\ZZ^d \cap F$ 
for some affine subspace $F\subset \RR^d$ of dimension $n$.
Assume $A$ and $B$ satisfy  
Condition \romanref{C:Ray}.
If there  exist affine subspaces $P\subset L$ of dimension $(l-1)$ and $l$
such that $P$ separates $A$ and $B$ in $L$,
then there exist 
affine subspaces  $P'\subset L'$ of dimension $l$ and $(l+1)$  
such that $P'$ separates $A$ and $B$ in $L'$.
\end{prop}

For readability we shall prove this proposition later.
It implies the following two lemmas, which are sufficient 
to prove Theorem \ref{thm:whole-lattice}.

\begin{lem}
\label{lem:a hyperplane sepelates fullspace}
Let $S=\ZZ^d \cap F$ for some affine subspace $F\subset \RR^d$.
If $A$ and $B$ satisfy
Condition \romanref{C:Ray},
then there exists a hyperplane $H \subset F$ 
such that $H$ separates $A$ and $B$ in $F$.
\end{lem}
\begin{proof}
Let $a\in A$ and $b \in B$.
For the line $L$ through $a$ and $b$,
some point $p \in L$ separates $A$ and $B$ in $L$
since $A$ and $B$ satisfy Condition \ref{C:Ray}.
Applying Proposition  \ref{lemma:inductionstep} inductively,
we have the lemma.
\end{proof}

\begin{lem}
\label{lemma:seq of hyper planes seterates fullspace}
Let $S=\ZZ^d \cap F$ for some affine subspace $F\subset \RR^d$ of dimension $n$.
If $A$ and $B$ satisfy Condition \romanref{C:Ray},
then Condition \romanref{C:Hyperplanes} holds.
\end{lem}
\begin{proof}
By Lemma \ref{lem:a hyperplane sepelates fullspace},
there exists a hyperplane $H_{n-1}$ in $H_n=F$
such that $H_{n-1}$ separates $A$ and $B$ in $H_n$.
Applying Lemma \ref{lem:a hyperplane sepelates fullspace} recursively
until $H_i \subset A$ or $H_i \subset B$,
we have a sequence of hyperplanes which separate $A$ and $B$.
\end{proof}

\begin{proof}[Proof of Theorem \ref{thm:whole-lattice}]
Take $F=\RR^d$ in Lemma \ref{lemma:seq of hyper planes seterates fullspace}.
\end{proof}

\subsection{Proof of Proposition \ref{lemma:inductionstep}}
It remains to prove Proposition \ref{lemma:inductionstep}.
The rest of this section is devoted to its proof in a series of lemmas.
For $1 \le i   \le d$, let 
\begin{align*}
 R_i (\theta) &= \Set{ 
(x_1, \ldots, x_{i-1}, t \cos \theta, t \sin \theta, 0, \ldots, 0) 
\in \RR^d  | x_1,\ldots,x_{i-1} \in \RR, t >0 }, \\
 H_i  &= \Set{ (x_1, \dots, x_i, 0, \dots, 0) \in \RR^d  |  x_1,\ldots,x_i \in \RR }, 
\end{align*}
$R_i^{+} = R_i(0)$ and $R_i^{-} = R_i(\pi)$. 
Assume $\varphi \in \GL(\RR^d)$ and $\alpha \in \RR^d$
satisfy the following:
$P=\varphi (H_{l-1}) + \alpha $,
$L=\varphi (H_{l}) + \alpha $,
$F=\varphi (H_{n}) + \alpha $,
$R^+=\varphi (R^+_{l}) + \alpha $,
$R^-=\varphi (R^-_{l}) + \alpha $,
$R^+ \cap S \subset A$ and
$R^- \cap S \subset B$.

\begin{lem}\label{lem2}
Let $a \in R^+\cap S = (\varphi (R_l^+) + \alpha) \cap S$. 
There exist $\theta'_a, \theta''_a,$ and $C_a$ such that
\begin{align*}
\theta'_a - \theta''_a &\ge \pi, 
& C_a &= 
\bigcup_{\theta''_a \le \theta \le \theta'_a} (\varphi (R_l(\theta)) + a), 
& \Int(C_a) \cap S &\subset A,
\end{align*}
where   
$\Int(X)$ denotes the relative interior of $X \subset \RR^d$.
\end{lem}

\begin{proof}
Define $\theta'_a, \theta''_a$ as
\begin{align*}
	\theta'_a &= \sup \Set{ \theta'  | \forall \theta \in [0,\theta'], \ (\varphi (R_l(\theta)) + a) \cap S \subset A }, \\
	\theta''_a &= \inf \Set{ \theta''  |  \forall \theta \in [\theta'',0], \ (\varphi (R_l(\theta)) + a) \cap S \subset A }.
\end{align*}
Since $(\varphi (R_l^{+}) + a) \subset (\varphi (R_l^+) +\alpha) \subset A$
 and $\emptyset \not= (\varphi (R_l^-) + \alpha) \subset B$, 
it follows that $-\pi \le \theta''_a \le 0 \le \theta'_a \le \pi$.
Let $C_a = \bigcup_{\theta''_a \le \theta \le \theta'_a} (\varphi (R_l(\theta)) + a)$. Then we have $\Int(C_a) \cap S \subset A$.

The proof is completed by showing that $\theta'_a - \theta''_a \ge \pi$. 
For convenience, let $a = 0$. 
If $\theta'_a - \theta''_a < \pi$, then
there exist $\theta',\theta''$ such that
\begin{gather*}
 \theta'' \le \theta''_a \le \theta'_a \le \theta' < \theta'' + \pi, \\
 \varphi (R_l(\theta'')) \cap B \not= \emptyset,  \\
 \varphi (R_l(\theta')) \cap B \neq \emptyset.
\end{gather*}
Let $b' \in \varphi (R_l(\theta')) \cap B, b'' \in \varphi (R_l(\theta'')) \cap B$. Since $b', b'' \in \ZZ^d$, there exist $t',t'' \in \ZZ$ such that
\begin{align*} 
 \frac{t'b' + t''b''}{t'+t''} \in \varphi (R_l^+). 
\end{align*}
Let $q' = (t'+t'')b', q'' = (t'+t'') b'',$ and 
$p = (t'q' + t''q'')/(t'+t'') = t'b' + t''b''$. 
Then $q', q'' \in B$, and $p \in (\varphi (R_l^+) + a) \cap \ZZ^d \subset A$.
However the three points $p,q',q''$ contradict  Condition \ref{C:Ray}.
\end{proof}

\begin{lem}
\label{lem3}
Let $b \in R^- \cap S = (\varphi (R_l^-) + \alpha) \cap S$. There exist $\theta'_b, \theta''_b,$ and $C_b$ such that
\begin{align*} 
\theta'_b - \theta''_b \ge \pi, \quad 
C_b = \bigcup_{\theta''_b \le \theta \le \theta'_b} (\varphi (R_l(\theta)) + b), \quad
\Int(C_b) \cap \ZZ^d \subset B.
\end{align*}
\end{lem}

\begin{proof}
The proof is the same as that of Lemma \ref{lem2}.
\end{proof}

Since $\Int(C_a) \cap S \subset A, \Int(C_b) \cap S \subset B$ and $A \cap B = \emptyset$, it immediately follows that 
$\theta'_a - \theta''_a = \pi$. 
Since 
\begin{align*} 
 C_a = \bigcup_{\theta''_a \le \theta \le \theta'_a} (\varphi (R_l(\theta)) + a) 
 = \varphi \left(\bigcup_{\theta''_a \le \theta \le \theta'_a} R_l(\theta)\right)
 + a,
\end{align*}
there exists $\psi_a \in \GL(\RR^d)$ such that 
$\psi_a (R_{l+1}^+) + a = \Int(C_a)$. 

Let us fix $a \in R^+\cap S$ and $\psi=\psi_a$.
Let $L'= \psi(H_{l+1})+a$ and $C'_{a'}=C_a-a+a' \subset L'$ 
for $a' \in L'\cap A$.
\begin{lem}\label{lem:2'}
For $a' \in L'\cap A$,
$\Int(C'_{a'}) \cap S \subset A.$
\end{lem}
\begin{proof}
By Lemma \ref{lem2},  $\Int(C_{a}) \subset A$.
If $a' \in C_{a}$, then $C'_{a'}\subset C_{a}$, which implies 
$ \Int(C'_{a'}) \subset \Int (C_{a}) \subset A$.
Let us consider the case when $a'\not\in C_{a}$.
For each $x\in \Int(C'_{a'}) \cap S$, there exists some point
 $q \in C_{a}\cap S \subset A$
on the ray from $a$ to $x$.
Since we assume Condition \ref{C:Ray}, $x$ is in $A$.
\end{proof}

Let 
\begin{align*}
R'_+ &= 
\bigcup_{a' \in L'\cap A} \Int(C'_{a'}) = 
\bigcup_{a' \in L'\cap A} \Int(C_{a}) -a +a' \\
&= \bigcup_{a' \in L'\cap A} (\psi (R_{l+1}^+) + a').
\end{align*}
Fix $\beta \in \partial R'_+ = \overline{R'_+} \setminus \Int(R'_+)$. 
Then we obtain 
\begin{align*}
	R'_+ &= \psi (R_{l+1}^+) + \beta, &
L' \setminus \overline{R'_+} &=
	\psi (H_{l+1}) \setminus \overline{R'_+} = \psi_a (R_{l+1}^-) + \beta.
\end{align*}
Let $R'_-=  \psi_a (R_{l+1}^-) + \beta$ and
$P'=L'\setminus ( R'_+ \cup  R'_-) $.
Then we turn to show that $P'$ separates $A$ and $B$ in $L'$.
\begin{lem}
\label{lem4}
$\emptyset \not= R'_- \cap S \subset B$.
\end{lem}

\begin{proof}
Since $A\cap L' \subset \overline{R'_+}$, 
it follows that $S\cap \left(L' \setminus \overline{R'_+} \right)\subset B$, which implies
$S \cap R'_- \subset B$.
\end{proof}

\begin{lem}\label{lem5}
$\emptyset \neq R'_+ \cap S \subset A$.
\end{lem}

\begin{proof}
For $x  \in R'_+ $, there exists  $a' \in L'$
such that $x \in \Int(C'_{a'})$, which implies $x \in A$.
\end{proof}

Now we have Proposition \ref{lemma:inductionstep}.  Note that
in the above proofs we described how to construct a sequence of separating
affine subspaces.  

\section{Separation of a finite set in dimension two}
\label{sec:finite}

In this section, we  consider the case when $S \subset \ZZ^d $ is finite. 
In the case $d=2$ simple separation results hold as shown in
the following theorems. 

\begin{thm}
\label{thm:finite}
Let $S \subset \ZZ^2$ be a finite integrally convex set. Then for $A \subset S$ and $B = S \setminus A$,  
Conditions \romanref{C:Hyperplanes} and \romanref{C:Parallelogram} are equivalent.
\end{thm}

\begin{thm}
\label{thm:finite2}
Let $S \subset \ZZ^2$ be a finite hole free set. Then for $A \subset S$ and $B = S \setminus A$,  
Condition \romanref{C:Hyperplanes} and 
the $3$-parallelogram condition are equivalent.
\end{thm}

We will prove these theorems in later subsections
by using the following lemma.

\begin{lemma}
\label{lemma:finite3}
Let $S \subset \ZZ^2$ be a finite hole free set. 
If $A$ and $B$ 
satisfy Condition \romanref{C:Parallelogram}, 
$A \setminus \conv(B) \neq \emptyset$ and $ B \setminus \conv(A) \neq \emptyset$,
then Condition \romanref{C:Hyperplanes} holds.
\end{lemma}

We show the following lemma before the proof of Lemma \ref{lemma:finite3}.
\begin{lemma}
\label{lemma:finite4}
Let $S \subset \ZZ^2$ be a finite hole free set. 
If $A$ and $B$ 
satisfy  Condition \romanref{C:Parallelogram}, then 
$\conv \set{a_1,a_2} \cap \conv \set{b_1,b_2} = \emptyset$ for all $a_1,a_2 \in A$ and $b_1,b_2 \in B$.
\end{lemma}

\begin{proof}
Let us assume that for some $a_1,a_2 \in A$ and $b_1,b_2 \in B$, $\conv \set{a_1,a_2} \cap \conv \set{b_1,b_2} \neq \emptyset$.  Consider $a_1, a_2, b_1, b_2$
such that 
$\conv \set{a_1,a_2,b_1,b_2} \cap \ZZ^2$ is minimal with respect to inclusion.
Thanks to the hole freeness of $S$,  we can 
suppose that 
$\conv \set{a_1,a_2,b_1,b_2} \cap \ZZ^2 = \set{a_1,a_2,b_1,b_2}$
without loss of generality.

Let $\delta = a_1 - a_2$ and $c = a_1 +a_2 - b_1$. Also let $c_+ = c + \delta, c_- = c - \delta$. 
Since the area of the triangle $\conv \set{a_1,a_2,b_1}$ is equal to that of 
the triangle $\conv \set{a_1,a_2,b_2}$, $b_2$ has to be on the line parallel to $\conv \set{a_1,a_2}$
through $c$. Also $\conv \set{a_1,a_2,b_1,b_2} \cap \ZZ^2 = \set{a_1,a_2,b_1,b_2}$ implies 
$b_2 \in \conv \set{c_+,c_-}$. Then since $\conv \set{a_1,a_2} \cap \ZZ^2 = \set{a_1,a_2}$, we have
\begin{align*} 
\conv \set{c_+,c_-} \cap \ZZ^2 = \Set{c,c_+,c_-}. 
\end{align*}
Thus $b' \in \set{c,c_+,c_-}$. However in any case we have a contradiction to Condition \romanref{C:Parallelogram}.
\end{proof}

\begin{proof}[Proof of Lemma \ref{lemma:finite3}]
If $\conv(A)\cap \conv(B)=\emptyset$, then Condition \romanref{C:Hyperplanes} 
holds by the usual hyperplane separation theorem.  Therefore it suffices
to show that $\conv(A)\cap \conv(B)\neq \emptyset$ leads to a contradiction.

Firstly, recall that the convex hull of a finite set of points is
equal to the union of simplices spanned by a finite set of the points \cite[Thm.2.15]{zie}. 
Therefore, if $\conv(A) \cap \conv(B) \neq \emptyset$,
there exist $a_1,a_2,a_3 \in A$ and $b_1,b_2,b_3 \in B$ such that 
$\conv\set{a_1,a_2,a_3}\cap \conv\set{b_1,b_2,b_3}\neq\emptyset$. 
By Lemma \ref{lemma:finite4}, it suffices to show that for some $a,a' \in A$ and $b,b' \in B$, $\conv \set{a,a'}$ 
and $\conv \set{b,b'}$ intersect.

Let $A' = \set{a_1,a_2,a_3}$ and $B' =  \set{b_1,b_2,b_3}$.
As the first case suppose that some $b \in B'$ is in $\conv(A')$. Then since $B \setminus \conv(A) \neq \emptyset$, 
there exists $b' \in B$ such that $b' \notin \conv(A')$. In this case, clearly the line segment $\conv \set{b,b'}$
crosses an edge of the triangle $\conv(A')$.

As the second case suppose 
that $B' \cap \conv(A') = \emptyset$ and $A' \cap \conv(B') = \emptyset$. 
The intersection of two triangles $\conv(A') \cap \conv(B')$ is a polygon.
Consider a vertex $v$ of the polygon. Since $v \in \conv(A') \cap \conv(B')$, $v$ is not a member of $A'\cup B'$.
Now each triangle is defined by three inequalities and  the intersection of the two triangles is
defined by effective inequalities of those six. 
Hence $v$ is the intersection of two edges 
among these  six edges. If $v$ is the intersection of two edges of one triangle,
then it reduces to the first case.
Hence $v$ is the intersection of
one edge of $\conv(A')$ and another edge of $\conv(B')$.
\end{proof}

In the finite case, note that Condition \ref{C:Ray} is not equivalent to Conditions \ref{C:Parallelogram} and
\ref{C:Hyperplanes}. 
\begin{ex}[A counterexample to  \ref{C:Ray}  $\Rightarrow$ \ref{C:Hyperplanes}] When 
$A = \{ (0,0),(1,1) \}$ and 
$B = \{ (1,0),(0,1) \}$, $S = A \cup B$ is integrally convex.
In this case, both $A$ and $B$ satisfy Condition \ref{C:Ray}. However we easily find that
Condistions \ref{C:Hyperplanes}  and \ref{C:Parallelogram} do not hold. 
\end{ex}

In the case $d=2$,
we have Theorems \ref{thm:finite} and \ref{thm:finite2}.
We want to consider more general cases or higher dimensions.
However, there are some counterexamples.

Firstly we consider Theorem \ref{thm:finite}. For general sets $S \subset \ZZ^2$,
the equivalence of  Theorem \ref{thm:finite} is violated as shown in the following examples.
\begin{ex}[Counterexamples to \ref{C:Parallelogram} $\Rightarrow$ \ref{C:Hyperplanes}]
We show some counterexamples in Figure \ref{fig:2-counter}. The points of $A$ are denoted by black circles, and 
those of $B$ by white circles. In these cases, Condition \ref{C:Parallelogram} holds  
(moreover Condition \ref{C:Ray} also holds from Lemma \ref{lem:c1=>c3=>c2}). 
However, $A$ and $B$ are not separated. 
\end{ex}
\begin{figure}
 \begin{center}
 \includegraphics[clip,width=0.2\textwidth]{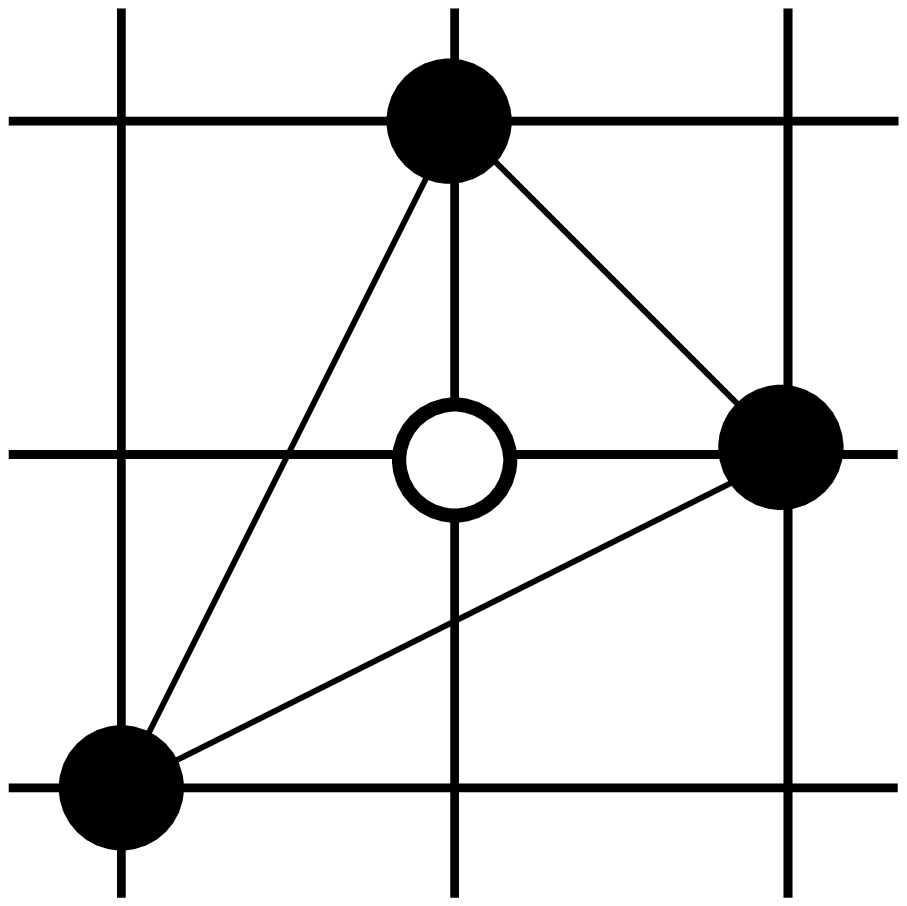}
    \label{fig:lattice2}
  \hfil
  \includegraphics[clip,width=0.2\textwidth]{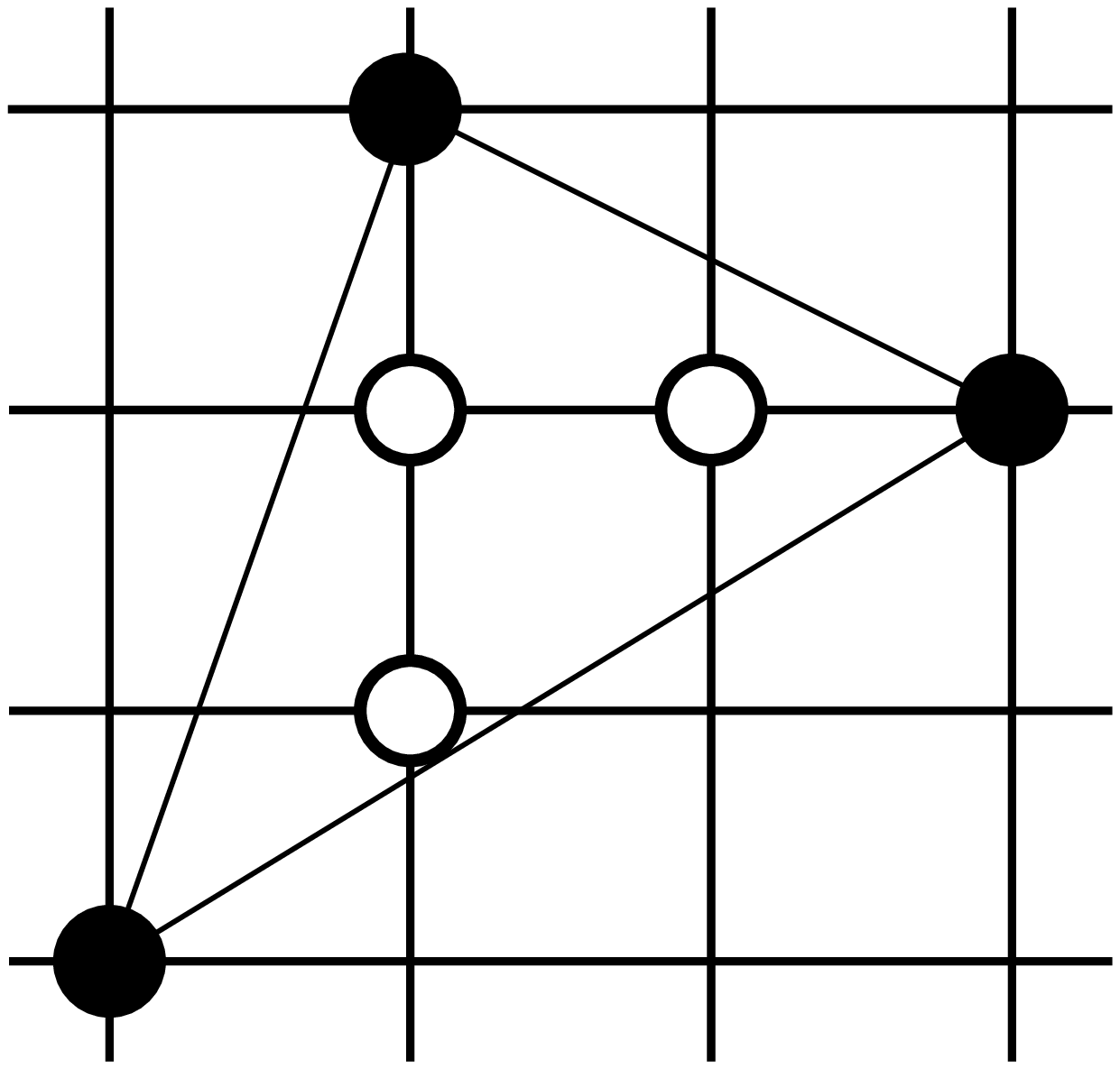}
    \label{fig:lattice4}
  \hfil
  \includegraphics[clip,width=0.35\textwidth]{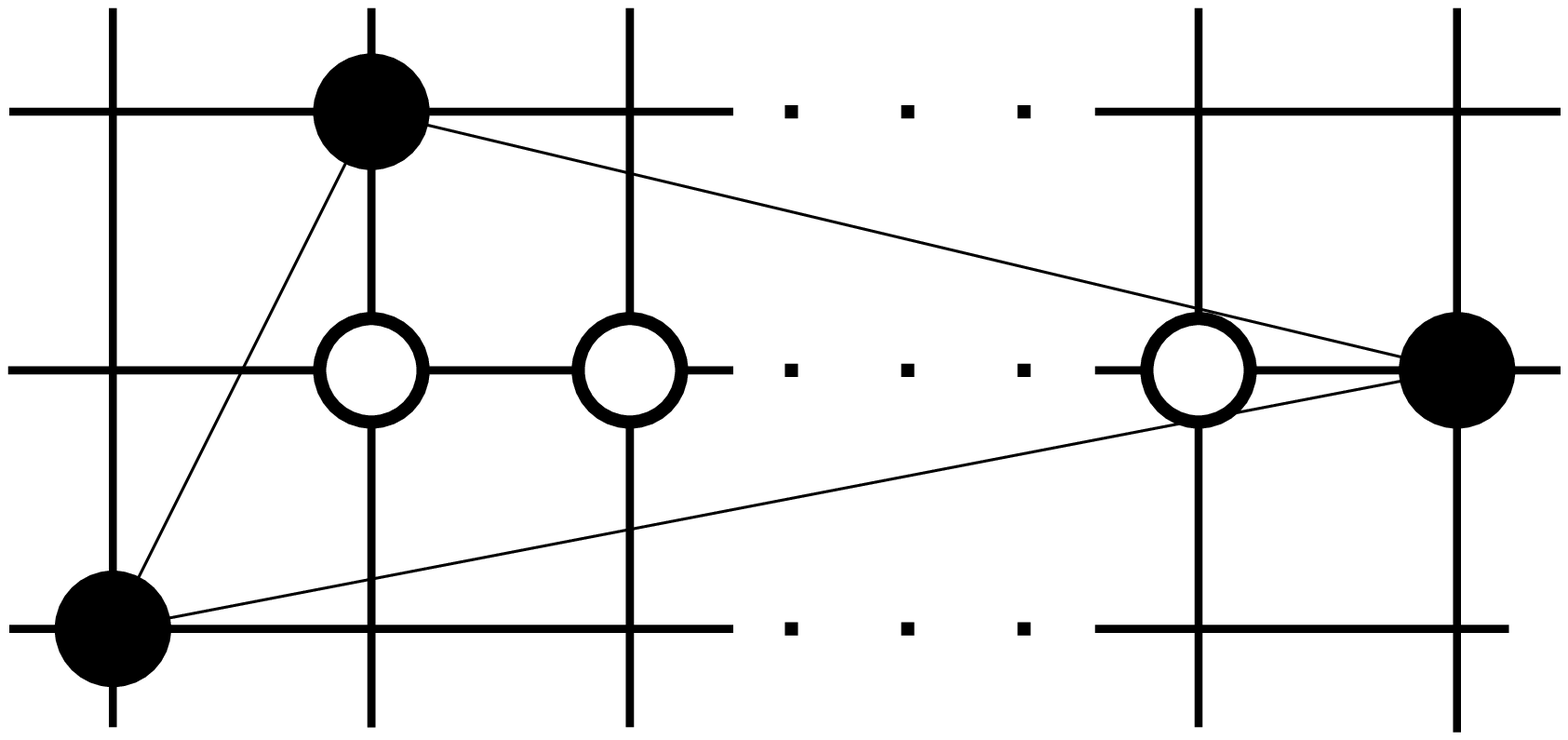}
    \label{fig:lattice3}
 \caption{Some counterexamples to \ref{C:Parallelogram} $\Rightarrow$ \ref{C:Hyperplanes} in $\ZZ^2$}
 \label{fig:2-counter}
 \end{center}
\end{figure}

We next consider the case $S \subset \ZZ^d$ for $d \ge 3$. 
In this case, by examples we show that 
Conditions \ref{C:Hyperplanes}, \ref{C:Parallelogram} and \ref{C:Ray} are not equivalent.
\begin{example}\label{ex:3dimension}
Consider 
\begin{align*}
S = \Set{ (x_1,x_2,x_3) \in \ZZ^3 | 0 \le x_1 \le 5,\, 0 \le x_2 \le 4,\, 0 \le x_3 \le 3 }. 
\end{align*}
Let $v_1, v_2, v_3$ denote
$ (5,0,0), (0,4,0), (0,0,3)$ and
$S' = \conv \{ 0,v_1,v_2,v_3 \} \cap \ZZ^3$
as in Figure \ref{fig5},  $A = S' \setminus \{ (2,1,1) \}$ and $B = S \setminus A$. Then 
$\conv(A) \cap \ZZ^3 = S'$ and it is clear
that there is no hyperplane separating $A$ and $B$. However we can easily check that Condition 
\ref{C:Parallelogram} holds.
\begin{figure}
\begin{center}
	\includegraphics[clip,width=0.5\hsize]{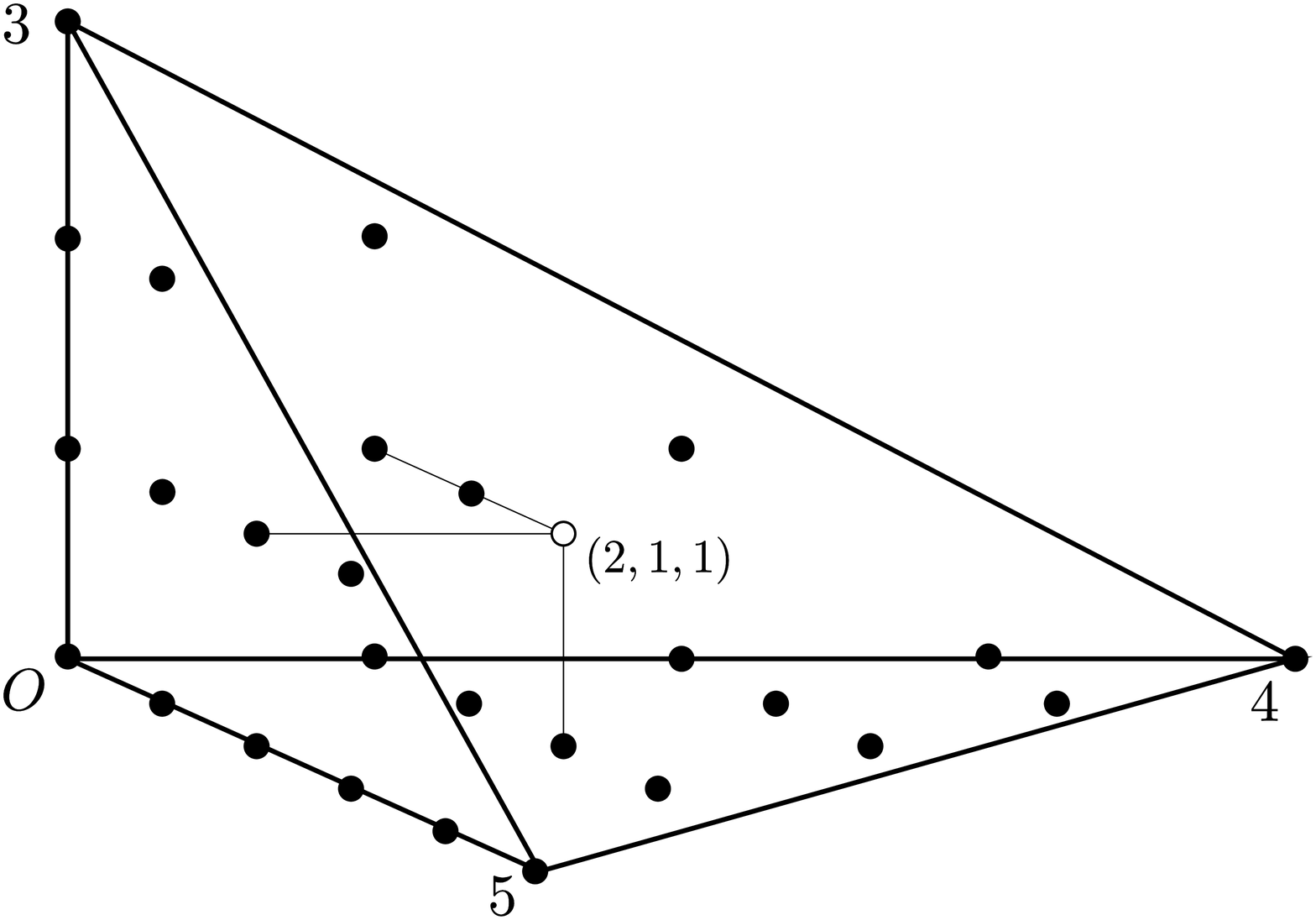}
	\caption{A counterexample in $\ZZ^3$}
	\label{fig5}
\end{center}
\end{figure}
\end{example}

In this example, we see that the set $A$ is the smallest $1$-convex set containing $\{ 0, v_1,v_2,v_3 \}$,
which we call the $1$-convex hull of $\{ 0,v_1,v_2,v_3 \}$. 
Similarly define the $k$-convex hull of $S$ as the smallest $k$-convex set containing $S$. 
In this case, since we have 
\begin{align*} 
\frac{1}{3}(5,0,0) + \frac{1}{3} (0,0,3) + \frac{1}{3} (1,3,0) 
= (2,1,1), 
\end{align*}
$S$ is $1$-convex but it is not $2$-convex (hence not $3$-convex). 
Therefore the $2$-convex hull of  $\{ 0,v_1,v_2,v_3 \}$ is the same as $S'$.

Then we are interested in the following example which shows that the
($d-1$)-convex hull is not necessarily $d$-convex.

\begin{example}
\label{ex:13-7-4}
Let $v_1, v_2, v_3$ denote the points
$(13,0,0), (0,7,0), (0,0,4)$ and  $S' = \conv \{ 0,v_1,v_2,v_3 \} \cap \ZZ^3$. 
Then by some detailed  calculation involving inner products, 
it can be shown that there exists no $a_1,a_2,a_3 \in S'$ such that 
\begin{align*} 
 (6,2,1) \in \conv \{ a_1, a_2 \}  \ \ \text{or} \ \ (6,2,1) \in \conv \{ a_1,a_2,a_3 \}. 
\end{align*}
Hence, the $2$-convex hull of $\{ 0,v_1,v_2,v_3 \}$ is not $3$-convex. 
\end{example}

This example shows that the equivalence of Theorem \ref{thm:finite2} is also violated for $d \ge 3$.

Thus we are interested in generalizations of Theorems \ref{thm:finite} and \ref{thm:finite2} using the $k$-parallelogram condition where $k \ge d$.
However the following counterexamples suggest that generalizations 
of Theorem \ref{thm:finite2} and Lemma \ref{lemma:finite3} are difficult 
for $d \ge 3$. 
\begin{ex} (An example based on a terminal Fano polytope in $\RR^3$) \\ 
\label{ex:d3-4par-1}
Let $v_1, v_2, v_3, v_4$ denote the points $(1,0,0), (0,1,0), (1,1,2)$, $(-1,-1,-1)$ and 
$S' = \conv \set{v_1,v_2,v_3,v_4}$. 
Then since 
\begin{align*} 
 v_1 + v_2 + v_3 + 2v_4 = 0, 
\end{align*}
we have $0 \in S'$. Let $A = S' \setminus \set{0}$ and $B = \set{0}$.
Since $S'$ has no lattice points other than $v_1, \dots, v_4,0$, we have 
the $4$-parallelogram condition.
However we can not separate $A$ and $B$. 
\end{ex}
\begin{ex}
\label{ex:d3-4par-2}
Let $a_1, a_2, a_3$ denote the points $(1,0,0), (0,1,0), (0,0,1)$ and
$b_1,b_2$ denote $(0,0,0), (1,1,2)$. 
Let $A = \set{a_1,a_2,a_3}$, $B = \set{b_1,b_2}$. Then $S = A \cup B$ is hole free and
$A \cap \conv(B) = B \cap \conv(A) = \emptyset$. 
We clearly have the $3$-parallelogram  condition, 
but not Condition \ref{C:Hyperplanes}.
\end{ex}

Despite the above counterexamples, 
we conjecture that a generalization of Theorem \ref{thm:finite} to  $d \ge 3$ holds (see Section \ref{sec:conclusion}). 

\subsection{Proof of Theorem \ref{thm:finite}}

For proving Theorem \ref{thm:finite}
we show the following basic lemmas on an integrally convex set.

\begin{lem}
For a line $L$ in $\RR^d$, 
if $L \cap \ZZ^d$ is integrally convex, then 
$L$ is of the form
\begin{align}
L = \Set{ x_0 + tp | t \in \RR} \label{eq:integrally line}
\end{align}
for some $x_0 \in \ZZ^d$ and $p \in \set{-1,0,1}^d \setminus \set{0}$.
\end{lem}
\begin{proof}
Let $I = L \cap \ZZ^d$. For an integer point $x_0 \in L \cap \ZZ^d$, choose $\alpha \in (-1,1)^d \setminus \set{0}$ such that $x_0 + \alpha \in L \setminus I$.
We denote $x = x_0 + \alpha$. Consider the intersection of $I$ and the integral neighborhood $N(x)$. Since $L = \conv(I)$ is a line, 
$I \cap N(x)$ contains two points, and one of them is $x_0$. Thus another one has to be 
$x_0 + p$ where $p \in \{ -1,0,1 \}^d \setminus \set{0}$. Then we have the lemma. 
\end{proof}
\begin{lem}
The faces of a finite integrally convex set in $\RR^d$ are integrally convex.
\end{lem}

\begin{proof}
Let $S \subset \ZZ^d$ be a finite integrally convex set. From Proposition \ref{prop:integrally2}, for $x \in \RR^d$, 
we can consider $\conv(S \cap N(x))$ as $\conv(S)$ restricted to $\conv(N(x))$. 
For a face $F$ of  $\conv(S)$ and
$x \in F$,  let $F_x = F \cap \conv(N(x))$. Since $F_x$ is a face of $\conv(S) \cap \conv(N(x))$,  
it suffices to show that $F_x \cap \ZZ^d$ is integrally convex. 

Let $P \subset \RR^d$ be a polytope, $\vvert(P)$ be the set of all vertices of $P$, and $E$ be a face of $P$.  Then 
\begin{align}
	P = \conv(\vvert(P)), \label{eq:zie1} \\
	\vvert(E) = E \cap \vvert(P). \label{eq:zie2}
\end{align} 
Since $\conv (S \cap N(x))$ is a 0/1-polytope, we have $\vvert(\conv(S \cap N(x)) = S \cap N(x)$.
Thus  by  \eqref{eq:zie2} we have
\begin{align*}
	\vvert(F_x) &= F_x \cap \vvert(\conv(S \cap N(x))) \\
				 &= F_x \cap S \cap N(x)
				= F_x \cap N(x).
\end{align*}
Therefore, for any $y \in F_x$, 
\begin{align}
	\conv(\vvert(F_x) \cap N(y))  &= \conv(F_x \cap N(x) \cap N(y)) \notag \\
											&= \conv(F_x \cap N(y)). \label{eq:proof1}
\end{align}
Since $F_x \cap \conv (N(y))$ is a face of $F_x$, we have
\begin{align}
	\vvert(F_x \cap \conv(N(y))) &= F_x \cap \conv(N(y)) \cap \vvert(F_x) \notag \\
											&= \vvert(F_x) \cap N(y). \label{eq:proof2}
\end{align}
Thus combining the equations \eqref{eq:zie1}, \eqref{eq:proof1} and \eqref{eq:proof2}, 
we see that $F_x \cap \ZZ^d$ is integrally convex.
\end{proof} 

From these lemmas, 
as an edge of  an integrally convex set in $\ZZ^2$ we only need to consider 
lines of the form \eqref{eq:integrally line}. 
We now finish our proof of  Theorem \ref{thm:finite}. 
In view of Lemma 4.3, suppose  that $B \subset \conv(A)$
or $A\subset \conv(B)$.   Without loss of generality let $B \subset \conv(A)$. 
Then 
\begin{align*} 
S = A \cup B \subset \conv(A) \cup B = \conv(A). 
\end{align*}
Since $S \supset A$, we have $\conv(S) = \conv(A)$. Thus
all vertices of $\conv(S)$ belong to $A$. 
This implies that the integer points of the boundary of $\conv(S)$ also belong to
$A$ by Condition \ref{C:Parallelogram}. Let $b \in B$. 
Above lemmas show that as we move from $b$ in a direction parallel to either axis, 
we reach an integer point of an edge of $S$, which belongs to $A$.
If we move to the opposite direction of $b$ we again reach an integer point 
of $A$. 
This contradicts Condition \ref{C:Parallelogram}.

\subsection{Proof of Theorem \ref{thm:finite2}}

Finally we give a proof of Theorem \ref{thm:finite2}.
When we look at counterexamples in Figure  \ref{fig:2-counter} again, we notice that
the centroid of black circles is also a centroid of white circles.  Therefore
these counterexamples are impossible if we additionally impose 
the $3$-parallelogram condition.

Let $S\subset \ZZ^2$ be hole free. 
Since Lemma \ref{lemma:finite3} holds, we only need to prove 
that the $3$-parallelogram condition implies
$A\cap \conv(B) = \emptyset$ and $\conv(A)\cap B = \emptyset$.
We prove it by contradiction. 
Note that we only need to 
consider minimal triangles in $\conv(A)$, 
i.e., triangles with no points from $A$ except
the vertices $a_1, a_2, a_3$. 
Moreover, each face of this triangle has no integer points except
them since we have the $2$-parallelogram condition.

Theorem \ref{thm:finite2} is a consequence of the following lemma.

\begin{lem}
\label{lem:traiange}
Let $A = \set{a_1,a_2,a_3}$.
Assume that $\conv \set{a_i,a_j} \cap \ZZ^2 = \set{a_i,a_j}$ for any $i,j \in \set{1,2,3}$ and $B = (\conv(A)\cap \ZZ^2)\setminus A$ is not empty.
Then there exist $b_1, b_2, b_3\in B$ such that $a_1 + a_2 + a_3 = b_1 + b_2 + b_3$.
\end{lem}
\begin{figure}
\begin{center}
	\includegraphics[clip, width =0.6\hsize]{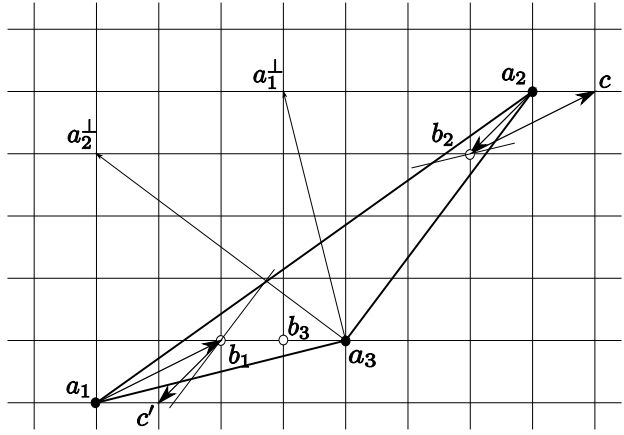}
\caption{A minimal triangle}
\label{fig:triangle}
\end{center}
\end{figure}
\begin{proof}
We can assume that $a_3 = 0$ without loss of generality.
Let us fix $a_1^{\perp}$, $a_2^{\perp}$ and $a_{12}^{\perp}$
such that 
\begin{align*}
&\Braket{a_1,a_1^{\perp}}=0,&
&\Braket{a_2,a_2^{\perp}}=0,&
&\Braket{a_1-a_2,a_{12}^{\perp}}=0,\\
&\Braket{a_2,a_1^{\perp}}>0,&
&\Braket{a_1,a_2^{\perp}}>0,&
&\Braket{-a_1,a_{12}^{\perp}}=\Braket{-a_2,a_{12}^{\perp}}>0.
\end{align*}
Then, for $x \in \RR^2$,
\begin{align*}
x\in \conv(A)
\iff 
\braket{x,a_1^\perp}\geq 0,
\braket{x,a_2^\perp}\geq 0,
\braket{x-a_1,a_{12}^\perp}\geq 0.
\end{align*}
By definition it also follows that 
$\braket{b-a_2,a_1^{\perp}} < 0$,
$\braket{b-a_1,a_2^{\perp}} < 0$ and
$\braket{b,a_{12}^{\perp}} < 0$
for all $b \in B$.
Let 
\begin{align*}
b_1&=\argmax_{b\in B} \braket{b-a_1,a_2^\perp}, \\
b_2&=\argmax_{b\in B} \braket{b-a_2,a_1^\perp}, \\
b_3&=a_1+a_2-b_1 -b_2.
\end{align*}
We show that $b_3\in B$, which concludes the lemma.
Let $c = b_1+b_2 -a_1$. 
Since $\Braket{c-a_2,a_1^\perp}=\braket{b_1-a_2,a_1^\perp}+\braket{b_2,a_1^\perp}>\braket{b_1-a_2,a_1^\perp}$,
by the definition of $b_2$, 
we get $c \notin \conv(A)$.
By direct calculation we obtain 
\begin{align*}
\braket{c,a_1^\perp}&=\braket{b_1,a_1^\perp}+\braket{b_2,a_1^\perp}>0,\\
\braket{c-a_1,a_{12}^\perp}&=\braket{b_1-a_1,a_{12}^\perp}+\braket{b_2-a_1,a_{12}^\perp}>0,
\end{align*}
which implies $\braket{c,a_2^\perp} < 0$.
Since $b_3=a_2-c$,
$\braket{b_3,a_2^\perp}=\braket{a_2-c,a_2^\perp}=-\braket{c,a_2^\perp} >  0$.
Similarly we also obtain 
$c' = b_1+b_2 -a_2 \not \in \conv(A)$ and
$\braket{b_3,a_1^\perp} > 0$.
Since
 $\braket{b_3-a_1,a_{12}^\perp}=\braket{b_3,a_{12}^\perp}+\braket{-a_1,a_{12}^\perp}>0$,
we have $b_3 \in \Int(\conv(A))$, which implies $b_3\in B$.
\end{proof}

\section{Conclusions and  open problems}
\label{sec:conclusion}
In this paper, we have proved the equivalence of Conditions \ref{C:Hyperplanes} and \ref{C:Parallelogram}
when $S=A\cup B$ is the whole $\ZZ^d$ or a 2-dimensional finite integrally convex set. 
However we have shown that it does not hold 
in $d$-dimensional finite sets for $d \ge 3$. 
For $S=\ZZ^d$ 
Condition \ref{C:Ray} is also equivalent to Conditions \ref{C:Hyperplanes} and \ref{C:Parallelogram}, 
but they are not equivalent in  the finite case. 
It is not clear what causes this difference.
We have also shown  the equivalence of  Condition \ref{C:Hyperplanes}
and 
the $3$-parallelogram condition for finite hole free $S\subset \ZZ^2$.

For $d\ge 3$ and finite $S$, we would like to obtain a useful condition which is equivalent to
Condition \ref{C:Hyperplanes}.  
Such a condition is relevant for extending the result of \cite{hara}
to higher dimensions.  
From various counterexamples in Section \ref{sec:finite}, extension
of our results for $d=2$ to higher dimensions may be difficult to prove.
In fact, in Examples 
\ref{ex:d3-4par-1} and \ref{ex:d3-4par-2}
we see that a natural generalization of Theorem \ref{thm:finite2} does not hold.
Nevertheless, we conjecture
that a generalization of Theorem \ref{thm:finite} holds in higher dimensions.
\begin{conjecture}
[Analogue of Theorem \ref{thm:finite}]
If $S$ is a finite
integrally convex set in $\RR^d$ and $A,B$ satisfy  
the $d$-parallelogram condition, then
Condition $H$ holds.  
\end{conjecture}

Another open problems concern the $k$-convexity of a set $A\in \ZZ^d$.
The $k$-convex hull of $A$ is 
a subset of the $(k+1)$-convex hull of $A$. 
A hole in $\conv(A)\cap \ZZ^d$ can be classified by
the first $k$ such that it belongs to the $k$-convex hull of $A$. 
For example, in Example \ref{ex:13-7-4}
with $A=\{(0,0,0)$, $(13,0,0)$, $(0,7,0)$, $(0,0,4)\}$
by computer search we found 
that the difference between the 2-convex hull and the 1-convex hull consists of 
a single point $(4,3,1)$.  $(6,2,1)$ discussed in 
Example \ref{ex:13-7-4} 
the difference between the 3-convex hull and the 2-convex hull.
It is of interest to classify holes by the $k$-convexity and give conditions
on existence or non-existence of holes according to this classification.

\begin{acknowledgment}
We are grateful to Kazuo Murota for very helpful comments.
\end{acknowledgment}


\begin{thebibliography}{99}

\bibitem{favat-tardella}
P.\ Favati and F.\  Tardella, Convexity in nonlinear integer programming,
{\it Ricerca Operativa} {\bf 53} (1990), 3--44.

\bibitem{hara} 
H.\  Hara, A.\ Takemura and R.\  Yoshida,
On connectivity of fibers with positive marginals in multiple logistic regression,
{\it Journal of Multivariate Analysis}
{\bf 101} (2010), 909--925.

\bibitem{hty}
R.\ Hemmecke, A.\ Takemura and R.\ Yoshida,
Computing holes in semi-groups and its applications to transportation problems,
{\it Contributions to Discrete Mathematics} {\bf 4} (2009),  
81--91.


\bibitem{kas06} 
A.\ M.\ Kasprzyk,
Toric Fano threefolds with terminal singularities,
{\it Tohoku Mathematical Journal (2)} {\bf 58} no.\ 1 (2006), 
101--121. 

\bibitem{kas08} 
A.\ M.\ Kasprzyk,
Canonical toric Fano threefolds,
preprint (2008).
{\tt arXiv:0806.2604v2}.
To appear in {\it Canadian Journal of Mathematics}.


\bibitem{murota98}
K.\ Murota, 
Discrete convex analysis,
{\it Mathematical Programming} {\bf 83} (1998), 
313--371.  

\bibitem{murota03} 
K.\ Murota, 
{\it Discrete Convex Analysis},
Monographs on Discrete Mathematics and Applications, 
SIAM, Philadelphia, 2003.

\bibitem{murota09}
K.\  Murota, 
Recent developments in discrete convex analysis, 
in: W.\ Cook, L.\ Lovasz and J.\ Vygen, eds.,
{\it Research Trends in Combinatorial Optimization, Bonn 2008},
Springer-Verlag, Berlin, 2009, Chapter 11, 219--260.

\bibitem{Ohsugi-Hibi-ample}
H.\ Ohsugi and  T.\ Hibi,
Non-very ample configurations arising from contingency tables,
preprint (2009).
{\tt arXiv:0904.3681v2}.
To appear in {\it Annals of the Institute of Statistical Mathematics}. 

\bibitem{sebo}
Andr\'as  Seb\H o, 
An introduction to empty lattice simplices.
in: {\it Integer Programming and Combinatorial Optimization}, 
Lecture Notes in Computer Science {\bf 1610},  400--414,  
Springer, Berlin,  1999. 

\bibitem{zie}
G.\ M.\ Ziegler, 
{\it Lectures on Polytopes},
Springer-Verlag, New York, 1995. 
\end{thebibliography}
\end{document}